\documentclass[a4paper,10pt]{article}

\usepackage{geometry}
\usepackage{amsmath}
\usepackage{amssymb}
\usepackage{amsthm}
\usepackage[latin1]{inputenc}
\usepackage{eurosym}
\usepackage[dvips]{graphics}
\usepackage{graphicx}
\usepackage{epsfig}
\usepackage{enumitem}
\usepackage{mdwlist}

\usepackage{color}

\usepackage{ifthen}

\newcommand{\argmax}{\operatorname{argmax}}
\renewcommand{\div}{\operatorname{div}}

\newcommand{\Rr}{{\mathbb{R}}}

\newcommand{\Pp}{{\mathcal{P}}}
\newcommand{\Ll}{{\mathcal{L}}}

\newcommand{\Fff}{{\mathcal{F}}}

\newcommand{\bx}{{\bf x}}

\newcommand{\bdx}{\dot{\bf x}}

\newcommand{\bp}{{\bf p}}
\newcommand{\bdp}{\dot{\bf p}}

\newcommand{\bX}{{\bf X}}
\newcommand{\bdX}{\dot{\bf X}}
\newcommand{\bY}{{\bf Y}}

\newcommand{\bP}{{\bf P}}
\newcommand{\bdP}{\dot{\bf P}}
\newcommand{\bR}{{\bf R}}

\newcommand{\bQ}{{\bf Q}}

\newcommand{\bv}{{\bf v}}

\newcommand{\bz}{{\bf z}}
\newcommand{\bdz}{\dot{\bf z}}

\newcommand{\bq}{{\bf q}}

\newtheorem{teo}{Theorem}
\newtheorem{df}{Definition}
\newtheorem{cor}{Corollary}
\newtheorem{lemma}{Lemma}
\newtheorem{remark}{Remark}
\newtheorem{pro}{Proposition}

\begin{document}

\title{Extended deterministic mean-field games}

\author{Diogo
A. Gomes\footnote{
King Abdullah University of Science and Technology (KAUST), CEMSE
Division, Thuwal 23955-6900. Saudi Arabia. e-mail: diogo.gomes@kaust.edu.sa} ,
Vardan K. Voskanyan\footnote{
King Abdullah University of Science and Technology (KAUST), CEMSE
Division, Thuwal 23955-6900. Saudi Arabia. e-mail vartanvos@gmail.com }
}

\date{\today} 

\maketitle

\begin{abstract}
In this paper, we consider mean-field games where the interaction of each player with the mean-field takes into account not only the states
of the players but also their collective behavior,
To do so, we develop a random variable framework that is particularly convenient for these problems.
We prove an existence result for extended mean-field games and
establish uniqueness conditions.
In the last section, we consider the Master Equation and discuss properties of its solutions.
\end{abstract}

\thanks{D. Gomes was partially supported by KAUST baseline and start-up funds and 
KAUST SRI, Center for Uncertainty Quantification in Computational Science and Engineering.}

\thanks{V.Voskanyan was supported
by KAUST baseline and start-up funds and 
KAUST SRI, Center for Uncertainty Quantification in Computational Science and Engineering.	
}

\section{Introduction}
\label{intro}

Mean-field games, MFGs, (see \cite{Caines2}, \cite{Caines1}
and \cite{ll1, ll2, ll3, ll4}) describe the behavior of systems
involving a large
number of rational agents who play dynamic games under partial information and symmetry assumptions.
For recent surveys on MFGs, see \cite{cardaliaguet, GS, llg2}.
In many applications, including crowd dynamics and economic problems, the behavior
of each agent depends on the statistical properties of the distribution of the agents and 
their collective actions. In the stationary case, such an extension was introduced and studied in \cite{GPatVrt}.
Here, we use a random variable formulation to examine this class of problems. 
These games are defined by the following system:
\begin{equation}
\label{mfge}
\begin{cases}
-u_t(x,t)+H( x, D_xu(x,t),\bX, \bdX)=0,\\
\bdX=-D_pH(  \bX, D_xu(\bX,t), \bX,\bdX),\\
u(x, T)=\psi(x, \bX(T)).
\bX(0)=X_0.
\end{cases}
\end{equation}
In the previous system, the unknowns are the value function $u(x,t)\colon \Rr^d\times [0,T]\to \Rr$ and a path in a space of random variables $\bX \colon [0, T]\to L^q(\Omega; \Rr^d)$,  where $(\Omega, \Fff, P)$ is a probability space and $1\leq q<\infty$.
The Hamiltonian $H\colon \Rr^d\times\Rr^d\times L^q(\Omega; \Rr^d)\times L^q(\Omega; \Rr^d)\to \Rr $, the terminal cost $\psi\colon\Rr^d\times L^q(\Omega; \Rr^d)\to\Rr$, and the initial state $X_0\in L^q(\Omega; \Rr^d)$ of the population are given. Detailed assumptions are presented in Section \ref{exist}. Two Hamiltonians for which those assumptions hold are
$$
H(x, p, X, Z)=\frac{|p|^2}{2}+\beta p EZ+V(x, X) \quad \text{ and }\quad H(x, p, X, Z)=\frac{|\beta E Z +p|^2}{2}+V(x, X),
$$
where $V\colon \Rr^d\times L^q(\Omega)\to \Rr$ is bounded,
$V(x,X)$ is twice differentiable in the $x$ variable, $|D_xV|$ and $|D^2_{xx}V|$ are uniformly bounded,  
and $V$ is Lipschitz in the variable $X$. Regarding the terminal cost, we suppose that $\psi(x,X)$ is
a continuous function, bounded by below, and uniformly Lipschitz in the first coordinate, $x$. 

Our main result is a proof of existence of solutions to extended mean-field games:
\begin{teo}
	\label{teoexist}
	Under Assumptions \ref{x0}-\ref{hjuniqc}(cf. Section \eqref{exist}), there exist a Lipschitz continuous function, $u\colon \Rr^d\times [0,T]\to\Rr$, and a path on the space of random variables, $\bX\in C^{1,1}([0,T],L^q(\Omega, \Rr^d))$, such that $(u,\bX)$ solves \eqref{mfge}. More precisely,
	$u\in C([0,T]\times\Rr^d; \Rr)$ is a viscosity solution of the Hamilton-Jacobi equation:
	\[
	\begin{cases}
	-u_t+H(  x, D_xu, \bX, \dot{\bX})=0,\text{ in } [0,T]\times\Rr^d,\\
	u(x, T)=\psi(x, \bX(T)), 
	\end{cases}
	\]
	where $u$ is $P$-a.s. differentiable at every point $(\bX(t),t),\, t>0$,
	and $\bX\in C^{1,1}([0,T];L^q(\Omega, \Rr^d))$ is a solution of the ODE:
	\[
	\begin{cases}
	\dot \bX=-D_pH(\bX,  D_xu(\bX,t), \bX,\dot{\bX}), \text{ in } [0,T]\times\Omega\\
	\bX(0)=X_0.
	\end{cases}
	\]
\end{teo}
The key difficulty in establishing the previous theorem is that, in general, the Hamilton-Jacobi equation in \eqref{mfge} does not admit classical solutions. 
Therefore, the right-hand side of the infinite dimensional ODE determining $\bdX$ in \eqref{mfge} is not locally Lipschitz. 
The proof of this theorem uses a new fixed-point argument. We observe that 
even for the original mean-field game problem, that is, if $H$ does not depend on $\bdX$, 
our results are not implied from the existing results in the literature;
for example see \cite{cardaliaguet} and Remark \ref{rmk2} at the end of Section 4. 

We conclude this introduction, by giving a brief outline of the paper. 
We begin Section \ref{formulation} with a concise discussion of the original formulation by Lasry and Lions of mean-field games as a transport equation coupled with a Hamilton-Jacobi equation. Then, in Section \ref{ranvar}, we develop a reformulation of this problem as an ordinary differential equation for a random variable in $L^q(\Omega)$ coupled with a Hamilton-Jacobi equation. This formulation is similar to the one used in \cite{Caines1}. These ideas were explored in \cite{ Carmona3, Carmona2, Carmona1} for problems 
with common noise. Next, in Section \ref{emfg}, we discuss extended MFGs, where we derive \eqref{mfge}. In addition, in Section \ref{exmp}, we present some examples for which the solutions can be found explicitly. 
In Section \ref{exist}, we prove the existence of solutions to the extended mean-field game system. Next, 
we discuss conditions for the absolute continuity of the law in Theorem \ref{abscont}. 
Concerning the uniqueness, we first consider a version of the Lasry-Lions monotonicity argument for classical solutions (Theorem \ref{unq1}) and an additional improvement for viscosity solutions (Theorem \ref{unq1'}). Then, we present the second approach to the uniqueness that uses the optimality nature of the solutions. We give a uniqueness condition (in terms of the Lagrangian (Theorem \ref{unq2}))
that does not require the absolute continuity of the law or regularity of the solution. Finally, in Section \ref{chp: furthprob}, we consider the Master Equation and discuss some of its properties.


\section{Two formulations of deterministic mean-field games}
\label{formulation}

In this section, we review the original formulation for deterministic mean-field games
from  \cite{ll1,ll2,ll3,ll4}. Then, we discuss a reformulation in terms of random variables. This set-up is very close to the one used in \cite{Caines1} (although they considered second-order equations) and it is particularly suited for the extensions we study here.

The standard mean-field game setting models a population in which each individual has a state $x\in \Rr^d$ and has access to the probability distribution  of the remaining players' states. 
We denote by $\Pp(\Rr^d)$ the set of probability measures in $\Rr^d$.
This set is a metric space endowed with the Wasserstein metric $W_2$, see, for instance, \cite{Villanithebook}.
At each time $t,$ the population is characterized  by a probability measure $\theta(t)\in \Pp(\Rr^d)$.
Each player seeks to  minimize a performance criterion. For this, let
$L:\Rr^d\times \Rr^d\times \Pp(\Rr^d)\to \Rr$ be a running cost, and
$\psi: \Rr^d\times \Pp(\Rr^d)\to \Rr$ be a terminal cost.
For definiteness, in this section, we suppose that both $L$ and $\psi$ satisfy standard hypotheses for optimal control problems, that is,
\begin{enumerate}[label=\alph*)]
\item\label{a}
$L$ and $\psi$ are continuous and bounded by below, without loss of generality, we can assume $L,\ \psi\geq 0.$
\item\label{b}
$\psi$ is Lipschitz in the first coordinate.
\item\label{c}
$L$ is coercive:
\[
\frac{L(x, v, \theta)}{|v|}\xrightarrow{|v|\to\infty}\infty, \text{ uniformly in } x.
\]
\item\label{d}
$L$ is uniformly convex in $v$.

\item\label{e}There exists a constant $C>0$, such that $L(x,0,\theta)\leq C$, $|D_xL(x,v,\theta)|, \ |D_vL(x,v,\theta)|\leq C\left[ 1+L(x,v, \theta) \right]$, for any $x,v\in\Rr^d$ and $\theta\in\Pp(\Rr^d).$
\end{enumerate}

An example that satisfies \ref{a}-\ref{e} is
\begin{equation}
\label{se}
L(x, v, \theta)=\frac{|v|^2}2-\int_{\Rr^d} V(x, y) d \theta(y),
\end{equation}
where $V\colon\Rr^d\times\Rr^d\to\Rr$ has a bounded $C^1$ norm.

We suppose 
each player directly controls his speed, that is, his state evolves according to $\bdx=v$. Assume that the distribution of the players is given by
a continuous curve
$\theta:[0,T]\to \Pp(\Rr^d)$. 
Fix a reference player, and let $x$ be the location of that player at time $t$. 
His value function is determined by the optimal control problem:
\[
u(x, t)=\inf_{\bx} \int_t^T L(\bx, \bdx, \theta(s)) ds+\psi(\bx(T), \theta(T)),
\]
where the infimum is taken over all Lipschitz trajectories with $\bx(t)=x.$
By Assumption \ref{a}, the map
$
(x,v,t)\mapsto L(x,v,\theta(t))
$
is continuous. Due to Assumptions \ref{c}-\ref{e}, it satisfies standard assumptions in control theory  (cf. \cite{FS}, Chapter I, sec. 9). Then,
Assumption \ref{b} gives that 
$u$ is bounded and globally Lipschitz.
For $( x, p,\theta)\in \Rr^d\times \Rr^d\times \Pp(\Rr^d),$ we
define the Hamiltonian
\[
H(x, p, \theta)=\sup_{v\in \Rr^d} -v\cdot p-L(x, v, \theta).
\]
 Under the above assumptions, $u$ is the unique globally Lipschitz and semiconvex viscosity solution of
the Hamilton-Jacobi equation
\begin{equation}
\label{hj}
-u_t+H(x, D_xu,  \theta)=0
\end{equation}
satisfying the terminal condition $u(x, T)=\psi(x, \theta(T))$. Moreover, if $u$ is a classical solution to \eqref{hj}, 
the optimal trajectories are given by
\begin{equation}
\label{feedbackc}
\bdx(s)=-D_pH(\bx(s), D_xu(\bx(s), s), \theta(s)).
\end{equation}
In mean-field games, all players have access to the same statistical information and act in a rational way. Therefore, each one of them follows the optimal trajectories \eqref{feedbackc}.  Consequently, the probability distribution of players is transported by the vector field
$-D_pH(x,  D_xu(x, t), \theta(t))$. Thus, in this MFG model, we require $\theta$ to be a (weak) solution
to the equation
\[
\theta_t-\div(D_pH(x,  D_xu,\theta)\theta)=0,
\]
satisfying the initial condition $\theta(0)=\theta_0\in \Pp(\Rr^d)$.
The probability measure $\theta_0$ encodes the distribution of players at $t=0$.
The above reasoning leads to the
system:
\begin{equation}
\label{mfgo}
\begin{cases}
-u_t+H(x,  D_xu, \theta)=0\\
\theta_t-\div(D_pH( x, D_xu,\theta)\theta)=0,
\end{cases}
\end{equation}
subjected to the initial and terminal conditions
\begin{equation}
\label{initcond}
\begin{cases}
u(x, T)=\psi(x, \theta(T))\\
\theta(x, 0)=\theta_0.
\end{cases}
\end{equation}
A second-order version of \eqref{mfgo}  was first introduced and studied  in \cite{ll2}. Detailed proofs of the existence and uniqueness of solutions to those systems can be  found in \cite{cardaliaguet}.
 Existence of smooth solutions of mean-field games for sub-quadratic and super-quadratic Hamiltonians with power-like and logarithmic local dependence on the population density has been  established in \cite{GPim1}, \cite{GPim2}, \cite{GPM3}, and \cite{GPM2}. Weak solutions have been considered in 
	\cite{porretta}, \cite{porretta2}, and \cite{cgbt}.

\subsection{Random variable framework}
\label{ranvar}

Let $(\Omega, \Fff, P)$ be a probability space, where $\Omega$ is an arbitrary nonempty set,
$\Fff$ is a $\sigma$-algebra on $\Omega$, and $P$ is a probability measure.
We recall that an $\Rr^d$-valued random variable $X$ is a $\Fff$-measurable function $X: \Omega\to \Rr^d$.
We denote by $L^q(\Omega, \Rr^d)$ the set of $\Rr^d$-valued random variables with a finite $q$-th moment, $E|X|^q<\infty$.
The law $\Ll(X)$ of an $\Rr^d$-valued random variable is the probability measure
in $\Rr^d$ defined by
\[
\int\limits_{\Rr^d} \varphi(x) d\Ll(X)(x)=E \varphi(X).
\]
Since all relevant random variables are $\Rr^d$ valued, we write $L^q(\Omega)$ instead of $L^q(\Omega, \Rr^d)$
to simplify the notation.
Next, we reformulate the mean-field game problem by replacing the probability $\theta(t)$ encoding the distribution of players by a
random variable $\bX(t)\in L^q(\Omega)$. The law of $\bX(t)$ determines the distribution of players, that is $\theta(t)=\Ll(\bX(t))$. Each outcome of the random variable $\bX$ gives the state of a particular player chosen accordingly to
the probability $\theta$. For each measure $\theta$, there is
an infinite number of random variables with law $\theta$. However, this ambiguity is harmless and does not create any technical difficulty. 

A function $f\colon L^q(\Omega)\to\Rr$ depends only on the law if for any $X, \tilde X\in L^q(\Omega)$ such that, $\Ll(X)=\Ll(\tilde X)$, we have $f(X)=f(\tilde X)$. Let
$\Pp_q(\Omega)$ be the set of probability measures $\theta\in\Pp(\Omega)$ with finite $q$-th moment: $\int_{\Rr^d}|x|^qd\theta(x)<+\infty$.
For $\eta: \Pp(\Rr^d)\to \Rr$, we define $\tilde \eta: L^q(\Omega; \Rr^d)\to \Rr$, by
\begin{equation}
\label{tildeq}
\tilde \eta(X)=\eta(\Ll(X)), 
\end{equation}
for $X\in L^q(\Omega; \Rr^d)$. Clearly $\tilde \eta$ depends only on the law of $X$.
Using this construction, we can identify functions in $\Pp_q(\Omega)$ with functions in $L^q(\Omega)$
that depend only on the law. 
Because there is no ambiguity, to simplify the notation, we omit the tilde in \eqref{tildeq} and
write without distinction $\eta(X)$ or $\eta(\Ll(X))$.

In this new setting, the Lagrangian is the function $L:\Rr^d\times \Rr^d\times L^q(\Omega)\to \Rr$, $L(x, v, X)$, and the terminal cost is the function $\psi: \Rr^d\times L^q(\Omega; \Rr^d)\to \Rr$, $\psi(x,X)$.
Again, for definiteness, we assume here that
\begin{enumerate}[label=\alph*$'$)]
	
	\item\label{a'} $L(x, v, X)$ and $\psi(x, X)$ depend only on the law of $X$.
	
	\item\label{b'}
	$L$ and $\psi$ are continuous functions in all variables and are bounded by below.
	\item\label{c'}
	$\psi$ is Lipschitz in the first coordinate.
	\item\label{d'}
	$L$ is coercive:
	\[
	\frac{L(x, v, X)}{|v|}\xrightarrow{|v|\to\infty}\infty, \text{ uniformly in } x.
	\]
	\item\label{e'}
	$L$ is strictly convex in $v$.
	\item \label{f'}
	There exists a constant $C>0$, such that $L(x,0,X)\leq C,$ $|D_xL(x,v, X)|, |D_vL(x,v, X)|\leq C\left(1+L(x,v, X)\right)$, for any $x,v\in\Rr^d$ and $X\in L^q(\Omega).$
\end{enumerate}

A Lagrangian that satisfies the above assumptions is
\[
L(x, v, X)=\frac{|v|^2}2- E V(x, X), 
\]
for $V:\Rr^d\times\Rr^d\to \Rr$ with bounded $C^1$ norm. This Lagrangian 
is the analog of \eqref{se}.

As before, suppose a given player knows the distribution of the remaining players. This distribution is determined by a path
$\bX\in C([0,T]; L^q(\Omega))$. He/she seeks to minimize a performance criterion, comprising a running cost
$L:\Rr^d\times \Rr^d\times L^q(\Omega; \Rr^d)\to \Rr$ and a terminal cost
$\psi: \Rr^d\times L^q(\Omega; \Rr^d)\to \Rr$.

The value function for
a reference player who is in the state $x$ at time $t$ is
\[
u(x, t)=\inf_{\bx} \int_t^T L(\bx, \bdx,\bX(s)) ds+\psi(\bx(T), \bX(T)).
\]
 For each trajectory $\bX\in C([0,T]; L^q(\Omega))$,
 by Assumption \ref{b'}, the function $(x, v, t)\mapsto L(x,v,\bX(t))$ is continuous. Furthermore, due to Assumptions \ref{c'}-\ref{f'}, it satisfies standard assumptions of the control theory (cf. \cite{FS}, Chapter I, sec. 9). Thus, $u$ is bounded and globally Lipschitz.
As before, for $( x, p,X)\in \Rr^d\times \Rr^d\times L^q(\Omega)$, the Hamiltonian
$H: \Rr^d\times \Rr^d\times L^q(\Omega)\to \Rr$
is given by
\[
H( x, p,X)=\sup_{v\in \Rr^d} -v\cdot p-L(x, v, X).
\]
The function $H( x, p, X)$ depends only on the law of the last coordinate. Accordingly,
if $X, \tilde X\in L^q(\Omega)$ have the same law (i.e., $\Ll(X)=\Ll(\tilde X)$) then
\[
H( x, p,X)=H( x, p,\tilde X).
\]
Under Assumptions \ref{b'}-\ref{f'},  $u$ is the unique bounded, globally Lipschitz and semiconvex viscosity solution of the Hamilton-Jacobi equation
\[
-u_t(x,t)+H( x, D_xu(x,t), \bX(t))=0
\]
with the terminal condition $u(x, T)=\psi(x, \bX(T))$.

Because of the rationality hypothesis,
a typical player $\omega\in \Omega$ has a trajectory determined by
\[
\bdX(s)(\omega)=-D_pH(\bX(s)(\omega), D_xu(\bX(s)(\omega), s), \bX(s)).
\]
This results in the system:
\[\label{mfgrva}
\begin{cases}
-u_t(x,t)+H( x, D_xu(x,t),\bX(t))=0\\
\dot \bX(t)(\omega)=-D_pH(\bX(t)(\omega), D_xu(\bX(t)(\omega),t), \bX(t)),
\end{cases}\tag{A}
\]
where the initial and terminal conditions \eqref{initcond} are replaced by
\[
\label{mfgrvb}
\begin{cases}
u(x, T)=\psi(x, \bX(T))\\
\bX(0)(\omega)=X_0(\omega),
\end{cases}
\tag{B}
\]
and $\Ll(X_0)=\theta_0$. Here, the transport equation in \eqref{mfgo} is replaced by an infinite-dimensional ODE in \eqref{mfgrva}.
To simplify, we omit $\omega$ in \eqref{mfgrva} and \eqref{mfgrvb}.
The connection between the two formulations is a consequence of the next well-known result:
Let $b\colon\Rr^d\times[0,T]\to\Rr^d$ be a bounded, continuous, uniformly Lipschitz in $x$  vector field over $\Rr^d$ and
let $\bX:[0,T]\times \Omega\to \Rr^d$ be a solution to
\begin{equation}
\label{forwardeq}
\bdX=b(\bX, t).
\end{equation}
Then, the law $\theta=\Ll(\bX)$ is a weak solution to
\[
\theta_t+\div(b\cdot \theta)=0,
\]
with the initial condition $\theta(0)=\Ll(\bX(0))$.

A proof of a stochastic version of this fact can be found in \cite{cardaliaguet} (Lemma 3.3). The proof of the present result follows along similar lines.



\section{Extended mean-field games}
\label{emfg}

In many applications, it is natural to consider mean-field games where the payoff of each player depends on the actions of the remaining players.
In the random variable framework, this corresponds to running costs that depend
on $\bdX$.
As before, we assume that the distribution of the players
is represented by a curve of random variables $\bX(t)\in L^q(\Omega)$. We require $\bX$ to be differentiable
with derivative $\bdX(t)\in L^q(\Omega)$.

Each player seeks to minimize a performance criterion. For this, let
$L:\Rr^d\times \Rr^d\times L^q(\Omega)\times L^q(\Omega)\to \Rr$ be a Lagrangian and 
$\psi: \Rr^d\times L^q(\Omega)\to \Rr$ be a terminal cost.
We assume that  $L(x, v, X, Z)$
depends only on the joint law of $(X, Z)$,  that is, if
$X, Z,\tilde X, \tilde Z\in L^q(\Omega)$ satisfy $\Ll(X, Z)=\Ll(\tilde X, \tilde Z)$ then
\[
L(x, v, X, Z)=L(x, v, \tilde X,\tilde Z). 
\]
Additionally, we require $\psi(x, X)$ to depend only on the law of $X$. To ensure
that the formulation of the problem 
makes sense, we need additional assumptions on $L$ and $\psi$. 
These are discussed in detail in the next section, however, for convenience, here, we outline the main requirements:
\begin{itemize}
\item[-]
$\psi$ is continuous and Lipschitz in the first coordinate (Assumption \ref{psi});
\item[-]
$L(x,v,X,Z)$ is coercive and uniformly convex in $v$ (Assumption \ref{lcnvcoer});	
\item[-]
$L$ is satisfy suitable continuity assumptions  and bounds by below (Assumptions \ref{JUY} and \ref{ab12}).
\end{itemize}
Additional assumptions are required
for existence and uniqueness results, 
which will be discussed later. 

The value function
for a player at state $x$ at time $t$ is
\begin{equation}
\label{controlf}
u(x, t)=\inf_{\bx} \int_t^T L(\bx, \bdx,   \bX(s), \bdX(s)) ds+\psi(\bx(T), \bX(T)),
\end{equation}
where the infimum is taken over all absolutely continuous trajectories $\bx:[t,T]\to \Rr^d$ with $\bx(t)=x$.
As before, the Hamiltonian $H: \Rr^d\times \Rr^d\times L^q(\Omega)\times L^q(\Omega)\to \Rr$ is given by
\[
H( x, p,X,Z)=\sup_{v\in \Rr^d} -v\cdot p-L(x, v, X, Z).
\]
We have that $H( x, p, X, Z)$ depends only on the joint law of the last two coordinates.

An important example is the Lagrangian
\begin{equation}
\label{standardL}
L( x, v,  X, Z)=\frac{|v|^2}{2}+\beta v\cdot E Z-V(x, X),
\end{equation}
to which corresponds the Hamiltonian
\begin{equation}
\label{standardH}
H( x, p, X, Z)=\frac{|\beta E Z +p|^2}{2}+V(x, X).
\end{equation}

If a player knows the trajectory $\bX(t)$ of the remaining players, the value function $u$ is determined by the Hamilton-Jacobi equation:
\[
-u_t+H(x, D_xu,\bX, \bdX)=0.
\]
Hence, by the mean-field hypothesis, each player follows the optimal trajectory determined by
\[
\bdx=-D_pH( \bx, D_xu(\bx,t),\bX, \bdX).
\]
Thus, we are led to the extended mean-field game
\begin{equation}
\begin{cases}
-u_t+H( x, D_xu,\bX, \bdX)=0\\
\bdX=-D_pH(  \bX, D_xu(\bX,t), \bX,\bdX),
\end{cases}
\end{equation}
with
\begin{equation}
\begin{cases}
u(x, T)=\psi(x, \bX(T))\\
\bX(0)=X_0.
\end{cases}
\end{equation}

\subsection{Examples}

\label{exmp}
The second equation in \eqref{mfge} involves a fixed-point problem for $\bdX$ that poses additional difficulties.
To ensure solvability, one possibility is to assume that
\begin{equation}
\label{contraction}
|D_pH( x, p, X, Z)-D_pH( x, p, X,  \tilde Z)|\leq \rho \|Z-\tilde Z\|_{  \ L^q(\Omega)},
\end{equation}
for some $\rho< 1$  and all $x,p\in \Rr^d, X, Z\in L^q(\Omega)$. In this last case, the equation

\begin{equation}
\label{velocity}
Z=-D_pH( X, P, X, Z),
\end{equation}
for  $X, P\in L^q(\Omega)$,
has a unique solution  $Z\in L^q(\Omega)$, by a standard contraction argument.

However, it may not be appropriate for all applications to impose such restrictive assumptions.
For instance,  the Hamiltonian \eqref{standardH} does not satisfy \eqref{contraction}.
Nevertheless, \eqref{velocity} becomes
\[
Z=-\beta E Z -P
\]
that has a unique solution $Z$, if $1+\beta\neq 0$. Indeed,
\[
Z+\beta EZ=-P.
\]
From this, we gather $EZ=-\frac 1 {1+\beta} EP$. Consequently,
$
Z=\frac \beta {1+\beta} EP-P.
$

In most cases,  \eqref{mfge} cannot be solved explicitly. In what follows, we present two examples with explicit solutions.

\subsection{Linear - Quadratic}
\label{linquad}

We consider the Hamiltonian
\[
H( x, p,X,Z)=\frac{|p+\beta EZ|^2}{2}+\frac 1 2 x^T A(X) x+B(X) \cdot x +C(X),
\]
where $A:L^q(\Omega)\to \Rr^{d\times d}$, $B: L^q(\Omega)\to \Rr^d$, and  $C: L^q(\Omega)\to \Rr$ are Lipschitz.
The term $\frac{|p+\beta EZ|^2}{2}$ in the Hamiltonian results in a corresponding term
$\frac{|v|^2}{2}-\beta v\cdot EZ$
in the Lagrangian.
For $\beta<0$, it penalizes players who move in the same direction as the aggregate
of the players, encoded here in $EZ$.   

We assume that the terminal condition is also quadratic in $x$, we
have
\[
\psi(x, X)=\frac{1}{2}x^TM(X) x+N(X)\cdot x+Q(X),
\]
with
$M:L^q(\Omega)\to \Rr^{d\times d}$, $N: L^q(\Omega)\to \Rr^d$, and  $Q: L^q(\Omega)\to \Rr$.

Due to the quadratic structure, we look for solutions of the form
\[
u(x, t)=\frac 1 2 x^T \Gamma(t) x +\Theta(t)\cdot x + \zeta(t).
\]
Using separation of variables, we obtain the system of differential equations
\[
\begin{cases}
-\dot \Gamma+\frac 1 2 \Gamma^T \Gamma+A(X)=0\\
-\dot \Theta+\beta \Gamma E\dot X+\Gamma \Theta +B(X)=0\\
-\dot \zeta+\frac 1 2 |\Theta+\beta E\dot X|^2+C(X)=0,
\end{cases}
\]
with terminal conditions
\[
\Gamma(T)=M(\bX(T)),\, \Theta(T)=N(\bX(T)),\, \zeta(T)=Q(\bX(T)).
\]
This system is coupled with the forward equation
\[
\begin{cases}
\bdX=-\Gamma \bX-\Theta-\beta E\bdX,\\
\bX(0)=X_0,
\end{cases}
\]
that is,
\[
\begin{cases}
\bdX=-\Gamma \bX-\frac 1 {1+\beta} \Theta+\frac{\beta}{1+\beta}(\Gamma E \bX)\\
\bX(0)=X_0.
\end{cases}
\]

\subsection{A second example }
\label{2exmp}

In addition to linear-quadratic Hamiltonians, we were able to find another class of problems, for which  \eqref{mfge} can be solved explicitly.
We consider the Lagrangian
\[
L(x, v, X,Z)=\frac{|v|^2}{2}+x^4+U(X,Z).
\]
This example features a quartic repulsive potential and allows for a arbitrary dependence on the mean-field through the function $U(X,Z)$, that we assume here to be continuous.
In contrast to the setting of Section \ref{formulation}, we assume here that players follow a distinct dynamic, namely: $\dot{\bx}=f(\bx, v),$ where $f(x,v)=\frac{v}{x}$.  Also,  we suppose that the terminal cost is of the form $\psi(x,X)=A(X)x^4+B(X)$.
Then, the Hamiltonian is
\[
H( x, p, X, Z)=\frac{|p|^2}{2x^2}-x^4-U(X,Z).
\]

Thanks to the distinctive structure of the Hamiltonian, $D_pH$ does not depend on $Z.$ So, it is trivial to solve the equation
$\bdX=-D_pH(\bX, \bP, \bX, \bdX)$. Moreover, the joint distribution of $(X,Z)$ gives an additive cost in the Lagrangian, identical for any player regardless of their location. 

The mean-field equations are
\[
\begin{cases}
-u_t+\frac{|Du|^2}{2x^2}-x^4-U(\bX,\dot{\bX})=0\\
\dot{\bX}=-\frac{Du(\bX,t)}{\bX^2}\\
u(x,T)=\psi(x,\bX(T)).
\end{cases}
\]
We look for solutions of the form $u(x,t)=x^4p(t)+q(t)$.
Note that the power $x^4$ is the only power for which we can use this separation of variables method.
This specific choice for $u$ gives
\[
\begin{cases}
p'-8p^2+1=0\\
q'=-U(X,\dot{X})\\
\dot{X}=-4pX\\
p(T)=A(\bX(T)),\, q(T)=B(\bX(T))\\
\bX(0)=X_0.
\end{cases}.
\]
Elementary computations yield
$$
p(t)=\frac{1}{2\sqrt{2}}\frac{1+ce^{4\sqrt{2}t}}{1-ce^{4\sqrt{2}t}},
$$
and
$$
\bX(t)=\left[\frac{ce^{4\sqrt{2}t}-1}{c-1}\right]^{-\frac{1}{2}}e^{\sqrt{2}t}X_0.
$$
Next, we determine the constant $c$ from the terminal condition on $p.$ Then
\[
u(x,t)=\frac{x^4}{2\sqrt{2}}\frac{1+ce^{4\sqrt{2}t}}{1-ce^{4\sqrt{2}t}}+q(t),
\]
where $q(t)$ solves $\dot{q}=-U(\bX,\dot{\bX})$ with $q(T)=B(\bX(T)).$

\section{Existence of solutions to extended mean-field games}
\label{exist}
Here, we study the existence of solutions to \eqref{mfge}. In what follows, $1\leq q< \infty$,
\[
\psi\colon\Rr^d\times L^q(\Omega)\to\Rr,
\]
\[
L\colon\Rr^d\times\Rr^d\times L^q(\Omega)\times L^q(\Omega)\to\Rr,
\]
and $H=L^*$ is the Legendre transform of $L$ defined by
\[
H( x, p, X, Z)=\sup\limits_v \{\,-v\cdot p-L(x, v, X,Z)\,\}.
\]
Moreover, 
we assume that $L(x, v, X, Z)$
depends only on the joint law of $(X, Z)$ and that $\psi(x, X)$ depends only on the law of $X$.

To prove existence of solutions to \eqref{mfge}, we 
require $X_0$, $\psi$, and $L$ to satisfy the following properties:

\begin{enumerate}[label=\Alph*)]
\item
\label{x0}
$X_0\in L^q(\Omega)$ and has an absolutely continuous law.
\item
\label{psi}
$\psi$ is Lipschitz continuous in $x$, and it is bounded.
\item
For any $x\in\Rr^d $ and $X,Z \in L^q(\Omega)$,
$L(x, v, X,Z)$ is uniformly convex in $v$ and satisfies the coercivity condition
$$
\lim_{|v|\to\infty}\frac{L(x, v, X,Z)}{|v|}=\infty,
$$
uniformly in $x$.
\label{lcnvcoer}
\item
\label{JUY}
 There exists a constant $c_0,$ such that $L(x, v, X,Z)\geq -c_0E(|X|^q+|Z|^q+1)$.
 \item 
\label{ab12}
There exist a constant $c_1>0$ and a continuous function $v_0\colon L^q(\Omega)\times L^q(\Omega)\to \Rr$, such that $L(x,v_0(X,Z),X,Z)\leq c_1$.
\item
\label{DL} There exist constants $c_2, c_3>0,$ such that
$|D_vL|$, $|D^2_{xv}L|$, $|D^2_{xx}L|$, $|D^2_{vv}L|\leq (c_2 L+c_3)E(|X|^q+|Z|^q+1)$,
and
$|D_xL|\leq c_2 L+c_3$.
\item
\label{solvb}
For any $X,Y,P\in L^q(\Omega)$, the equation $Z=-D_pH( X, P,Y,Z)$ can be solved with respect to $Z$ as
$$
Z=G( X, P,Y).
$$
Moreover, the map $G\colon L^q(\Omega)\times L^q(\Omega)\times L^q(\Omega)\to L^q(\Omega)$ is Lipschitz.
\item
\label{Hcont}
$H$ is continuous in $X,Z$ locally uniformly in $x,p$.
\item
\label{dxHlip}
 $D_xH$ is Lipschitz in $\Rr^d\times\Rr^d\times L^q(\Omega)\times L^q(\Omega).$

\item
\label{hjuniqc}
For any $R>0$ there exists a constant $C(R)$ such that
 \[|H(x_1, p_1, X_1, Z_1)-H(x_2, p_2, X_2, Z_2)|\leq C(R)(|p_1-p_2|+|x_1-x_2|+
 W_q(\Ll(X_1, Z_1), \Ll(X_2,Z_2)).
 \]
 for $|p_i|,|x_i|,\|X_i\|_{L^q(\Omega)},\|Z_i\|_{L^q(\Omega)}\leq R,\ i=1,2.$ 
Here, $W_q$ denotes the $q$-Wasserstein metric on $\mathcal{P}(\Rr^d)$ (cf. \cite{Villanithebook}).
\end{enumerate}

From Assumption \ref{hjuniqc}, we have 
\[
|H(x_1, p_1, X_1, Z_1)-H(x_2, p_2, X_2, Z_2)|\leq C(R)(|p_1-p_2|+|x_1-x_2|+
 \|X_1-X_2\|_{L^q(\Omega)}+\|Z_1-Z_2\|_{L^q(\Omega)}). 
 \]

Two examples that satisfy the above conditions are:
\begin{equation}
\label{lagrangians}
L(x, v, X,Z)=\frac{|\beta E Z +v|^2}{2}-V(x, X)\text{ and }\, L(x, v, X,Z)=\frac{|v|^2}{2}+\beta vEZ-V(x, X),
\end{equation}
with $V\colon\Rr^d\times L^q(\Omega)\to\Rr^d$  bounded, $C^{\infty}$ in the $x$ variable  with derivative $D_xV$ Lipschitz in $X$. 
Quadratic Lagrangians are standard in optimal control. The extra dependence on $Z$  indicates that players prefer to move with a velocity close to $-\beta EZ$ (i.e., in the opposite direction of the average of the population). The corresponding Hamiltonians are
\begin{equation}\label{hamiltonians}
H( x, p, X, Z)=\frac{|p|^2}{2}+\beta p EZ+V(x, X)\text{ and } H( x, p, X, Z)=\frac{|\beta E Z +p|^2}{2}+V(x, X).
\end{equation}
Assumption \ref{solvb} can be checked explicitly in the examples above.
For both Lagrangians, the function $G$ is given by
$G(X,P,Y)=-P+\frac{\beta}{1+\beta}EP.$ We also note that by the inverse function theorem, suitably small perturbations of the Lagrangians in these examples also satisfy assumption \ref{solvb}.

For each Lipschitz continuous function $\Phi\in C^{0}(R^n)$, consider the system of ODEs in $L^q(\Omega)$
\begin{equation}
\label{hjfl}
\begin{cases}
\bdX=-D_pH(\bX, \bP, \bX,\bdX)\\
\bdP=D_xH(\bX, \bP, \bX,\bdX)\\
\bX(0)=X_0,\,\bP(0)=D_x\Phi(X_0).
 \end{cases}
\end{equation}
 Using Assumption \ref{solvb}, this system can be rewritten as
\begin{equation}
\label{hjfl2}
\begin{cases}
\dot{\bX}=G(\bX, \bP, \bX,)\\
\dot{\bP}=D_xH(\bX, \bP, \bX,G(\bX, \bP, \bX,))\\
\bX(0)=X_0,\,\bP(0)=D_x\Phi(X_0).
 \end{cases}
\end{equation}

Since $\Phi$ is Lipschitz, $D_x\Phi$ exists almost everywhere. Additionally, because $X_0$ has an absolutely continuous law, $\bP(0)$ is well defined on a set of full measure.
Assumptions \ref{dxHlip} and \ref{solvb} on $G$ and $D_xH$ imply that the right-hand sides of the equations in \eqref{hjfl2} are Lipschitz in $\bX$ and $\bP$.
Furthermore, it has uniform linear growth.  Consequently, a Picard-Lindelof-type argument ensures the existence and uniqueness of global solutions of \eqref{hjfl2} and hence to \eqref{hjfl}. 
Standard arguments imply that $\bX,\bP\in C^1([0,T];L^q(\Omega))$ and, additionally, $\bdX$ is Lipschitz in $t$, that is, $\bX\in C^{1,1}([0,T];L^q(\Omega))$.

\begin{lemma}
\label{c8}
Let $\Phi\in C^0(\Rr^d)$ be a bounded Lipschitz continuous function with Lipschitz constant $Lip(\Phi)\leq R$. Let $\bX\in C^{1,1}([0,T];L^q(\Omega))$ be the random variable determined by \eqref{hjfl}. Then there exists a constant $c_4=c_4(L,R,X_0,T,q)$ such that
\[
E|\bX(t)|^q+|\dot \bX(t)|^q\leq c_4,\quad 0\leq t\leq T.
\]
\end{lemma}

\begin{proof}
From \eqref{hjfl2}, since $G\text{ and }D_xH$ are Lipschitz, we have
\begin{equation}
\label{ineq}
\begin{cases}
E|\dot{\bX}|^q\leq CE(|\bX|^q+|\bP|^q)+C\\
E|\dot{\bP}|^q\leq CE(|\bX|^q+|\bP|^q)+C.
\end{cases}
\end{equation}
Combining the two estimates in \eqref{ineq} yields
\[
E(|\dot \bX|^q+|\dot \bP|^q)\leq CE(|\bX|^q+|\bP|^q)+C,
\]
where $C$ depends on only $G$ and $D_xH$.
Let $\bR=(\bX,\bP)$ and $\|\bR\|=\left[E(| \bX|^q+| \bP|^q)\right]^{\frac{1}{q}}$. The inequality above gives $\|\dot \bR\|\leq C+C\|\bR\|$.
From Gronwall's inequality, we have
\[
\bR(t)\leq C(T)(1+\|\bR(0)\|),\quad\forall t\in[0,T].
\]
Therefore,
\[
E(|\bX|^q+|\bP|^q)\leq C(T)(1+E(|X_0|^q+|D\Phi(X_0)|^q))\leq C(T)(1+E|X_0|^q+|Lip(\Phi)|^q).
\]
The previous estimate together with \eqref{ineq} yields the required result.
\end{proof}

Next,  suppose $(\bX,\bP)$ is a solution of \eqref{hjfl}. We define $\tilde{u}(x,t)$ to be the solution to the optimal control problem
\begin{equation}
\label{utild}
\tilde{u}(x,t)=\inf\limits_{\bx}\int\limits_t^TL(\bx,\dot{\bx},\bX,\dot{\bX})+\psi(\bx(T),\bX(T)),
\end{equation}
where the infimum is taken over all absolutely continuous trajectories $\bx(s)$, with $\bx(t)=x$.

\begin{lemma}
\label{blsmc}
Let $\Phi\in C(\Rr^d)$ be any bounded, Lipschitz function and let $(\bX,\bP)$ be a solution of \eqref{hjfl}. Then,
$\tilde{u}(x,t)$ defined by \eqref{utild} is uniformly bounded and Lipschitz in $x$. Furthermore, for any $t<T$, $\tilde{u}$ is semiconcave in $x$. More specifically, there exists constants $c_5$ and $c_6$, such that $c_5$ depends only on $L, \psi$, $T$ and $c_6$ depends only on $L$, $Lip(\Phi)$, $T$ and $T-t_1$, and
\begin{enumerate}
\item $\tilde{u}\leq c_1(T-t)+\|\psi\|_{\infty} \text{ for all } x\in \Rr^d, 0\leq t\leq T.$
\item $|\tilde{u}(x+y,t)-\tilde u(x,t)|\leq  c_5|y|\text{ for all } x,y\in \Rr^d, 0\leq t\leq T.$
\item $\tilde{u}(x+y,t)+\tilde u(x-y,t)-2\tilde u(x,t)\leq c_6|y|^2\text{ for all } x,y\in \Rr^d,\quad 0\leq t\leq t_1<T$
\end{enumerate}

\end{lemma}

\begin{remark}
	The above lemma is classical, and similar results can be found in \cite{Bardi} (Theorem 4.9, p. 69) or \cite{FS} (c.f. the discussion in IV.9, p. 186).
	For our purposes, we need to ensure the uniformity of the constants in the data of the problem, namely the explicit dependence on the norms of $\bX$ and $\bdX$ is essential to use a fixed point theorem.
	Hence, we present here a detailed proof in the Appendix.
\end{remark}

Let $\overline{c_5}=c_5,\overline{c_6}=c_6$ be the constants from Lemma \ref{blsmc} for $t_1=0$, and
\[
\overline{c_7}=\max\{\,Tc_0(c_4(L,c_5,X_0,T,q)+1)+\|\psi\|_{\infty},Tc_1+\|\psi\|_{\infty}\,\}.
\]
We denote by $\mathcal{A}$ the set of functions $\Phi\in C(\Rr^d)$, with $|\Phi|\leq \overline{c_7},$ $Lip(\Phi)\leq \overline{c_5}$ and $\Phi$ semiconcave with the constant $\overline{c_6}.$
\begin{lemma}
\label{Fcont}
The mapping
$$
F\colon \Phi(\cdot)\longmapsto \tilde{u}(\cdot,0)
$$
is a continuous compact mapping  from $\mathcal{A}$ into itself (with respect to the topology of locally uniform convergence).
\end{lemma}

\begin{proof}
First,
we show that $F$ maps the set $\mathcal{A}$ into itself. According to the estimate from Lemma \ref{c8}, we get
\begin{equation*}
\begin{split}
&\tilde{u}(x,0)=\inf\limits_{\bx}\int\limits_0^TL(\bx,\dot{\bx},\bX,\bdX)+\psi(\bx(T),\bX(T)\geq -c_0\int\limits_0^TE(|\bX|^q+|\dot \bX|^q+1)-\|\psi\|_{\infty}\\&\geq -Tc_0\left(c_4(L,c_5,X_0,T,q)+1\right)-\|\psi\|_{\infty}\geq-\overline{c_7}.
\end{split}
\end{equation*}
The previous identity, combined with Lemma \ref{blsmc}, implies that $\tilde u\in \mathcal{A}$.
To prove the continuity of the mapping $F$ we argue by contradiction. Suppose there exists $\Phi_n\to \Phi$ in $C(\Rr^d)$ such that $F(\Phi_n)\nrightarrow F(\Phi)$ in $C(\Rr^d)$. Then, since $F(\Phi_n)\in \mathcal{A}$ and $\mathcal{A}$ is compact, we can assume, without loss of generality, that $\tilde{u}_n=F(\Phi_n)\to\bar{\Phi}\neq F(\Phi)$, locally uniformly. Because $\Phi_n$ are uniformly semiconcave, we can assume that $D\Phi_n\to D\Phi$ almost everywhere. We have that the corresponding trajectories $(\bX_n,\bP_n)$ solve
\begin{equation*}
\begin{cases}
 \dot{\bX}_n=G(\bX_n, \bP_n, \bX_n)\\
\dot{\bP}_n=D_xH(\bX_n,\bP_n, \bX_n, G(\bX_n, \bP_n, \bX_n))\\
\bX_n(0)=X_0,\,\bP_n(0)=D_xu_n(X_0).
 \end{cases}
\end{equation*}
By Gronwall's inequality,
\begin{equation*}
\begin{split}
&E(|\bX_n(t)-\bX(t)|^q+|\bP_n(t)-\bP(t)|^q)\leq C (E|\bX_n(0)-\bX(0)|^q+|\bP_n(0)-\bP(0)|^q)=\\&C E|D\Phi_n(X_0)-D\Phi(X_0)|^q.
\end{split}
\end{equation*}
The dominated convergence theorem entails that the right-hand side of the preceding expression converges to zero. Consequently,  $\bX_n\to \bX$ and $\bP_n\to \bP$ in $L^{\infty}([0,T];L^q(\Omega))$. Equation \eqref{hjfl2} implies that $\dot{\bX}_n\to\dot{\bX}$. Thus, using Assumptions
\ref{Hcont} and \ref{psi},
$$
H( x, p,\bX_n, \dot{\bX}_n)\to H(p, x,\bX, \dot{\bX}) \text{ locally uniformly in } x,p,
$$
and
$$
\psi(x,\bX_n(T))\to\psi(x,\bX(T))\text{ locally uniformly in } x.
$$
Because $\tilde{u}_n\to\bar{\Phi}$ locally uniformly, the stability of viscosity solutions (cf \cite{FS}) implies that $\bar{\Phi}$ is a viscosity solution of the Hamilton-Jacobi equation
\begin{equation}
\label{hjt}
\begin{cases}
-\bar{u}_t(x,t)+\tilde{H}(x, D_x\bar{u}(x,t), t)=0,\\
\bar{u}(x,T)=\psi(x,\bX(T)),
\end{cases}
\end{equation}
with Hamiltonian
\[
\tilde{H}( x, p,t)=H( x, p,\bX(t), \bdX(t)).
\]

On the other hand, the definition of $\tilde{u}=F(\Phi)$ implies that $\tilde{u}$  is a viscosity solution of \eqref{hjt}.

Since $\bX$ and $\dot \bX$ are Lipschitz continuous in $t$, Assumption \ref{hjuniqc} implies
 \[|\tilde{H}( x, p,t)-\tilde{H}(y, p', s)|\leq C(R)(|p-p'|+|x-y|+|t-s|),\]
for any $R>0$ and $t,s\in[0,T]$ and all $x,y,p,q \in\Rr^d$ with $|p|,|p'|\leq R$. This condition
gives the uniqueness of the viscosity solutions to \eqref{hjt}(\cite{FS}), thus $\bar{\Phi}=\tilde{u}$. This establishes the contradiction. Hence the mapping $F$ is continuous. Finally, the compactness of $\mathcal{A}$ implies that $F$ is compact.
\end{proof}

\begin{proof}[Proof of Theorem \ref{teoexist}]
The set $C(\Rr^d)$, endowed with the topology of locally uniform convergence, is a topological vector space. Moreover, $\mathcal{A}$ is a compact convex subset. Thus, by Lemma \ref{Fcont} and Schauder's fixed-point theorem, there exists $\Phi\in C(\Rr^d)$ such that
$$\Phi(\cdot)=F(\Phi)=\tilde{u}(\cdot,0),$$
where $\tilde{u}(x,t)$ is defined as in \eqref{utild}.
Let $u(x,t):=\tilde{u}(x,t).$ Then, $u$ solves the Hamilton-Jacobi equation
\begin{equation}
\label{Hj}
\begin{cases}
-u_t(x,t)+H( x, D_xu(x,t),\bX(t),\dot{\bX}(t))=0\\
u(x,T)=\psi(x,\bX(T)).
\end{cases}
\end{equation}
From standard results in optimal control theory (see \cite{FS}), for almost every $x,$ there exists an optimal trajectory given by the Hamiltonian flow
\begin{equation}
\label{Hjfl}
\begin{cases}
\dot{\bf x}(x,t)=-D_pH({\bf x}, {\bf p}, \bX, \dot{\bX})\\
\dot{\bf p}(x,t)=D_xH({\bf x}, {\bf p}, \bX, \dot{\bX})\\
{\bf x}(x,0)=x,\,{\bf p}(x,0)=D_xu(x,0).
\end{cases}
\end{equation}
We also know that ${\bf p}(x,t)=Du(\bx(x,t),t)$ and that $Du$ exists at all points $({\bf x}(x,t),t)$ with $t>0.$ Set $\bY(t)={\bf x}(X_0,t)$ and $\bQ(t)={\bf p}(X_0,t)$. Then, by \eqref{Hjfl}, we have
\begin{equation}
\begin{cases}
\label{hjf}
\dot{\bY}=-D_pH(\bY, \bQ, \bX, \dot{\bX})\\
\dot{\bQ}=D_xH(\bY, \bQ, \bX, \dot{\bX})\\
\bY(0)=X_0,\,\bQ(0)=D_xu(X_0).
\end{cases}
\end{equation}
Since $D_pH,D_xH$ are Lipschitz in $p,x$, the uniqueness of solutions to the system of ordinary differential equations in $L^q(\Omega)$,\eqref{hjf}, yields $\bX(t)=\bY(t)$ and $\bP(t)=\bQ(t)=\bp(X_0,t)=Du(\bY(t),t)$, for all $t\in[0,T]$.
In this way
\begin{equation*}
\label{Tsp}
\begin{cases}
\dot{\bX}(t)=-D_pH(\bX(t), Du(\bX(t),t),\bX(t), \dot{\bX}(t))\\
\bX(0)=X_0.
\end{cases}
\end{equation*}
\end{proof}

\begin{remark}
\label{rmk2}
This existence proof does not require
the absolute continuity of the law of $\bX$,
only the absolute continuity of the law of the initial condition $X_0$.
Accordingly, the preceding theorem extends the results
in \cite{cardaliaguet}.
\end{remark}

\section{Absolute continuity of the law of trajectories}

In the present section, we give conditions under which solutions to extended mean-field games have an absolutely continuous law. Our techniques are related to the ones in
\cite{cardaliaguet}, where the case of quadratic Hamiltonians was discussed.

Consider a Hamiltonian $H^\sharp\colon \Rr^d\times\Rr^d\times[0,T]\to\Rr$ and the corresponding Lagrangian $L^\sharp$ satisfying the following conditions:
\begin{enumerate}
\item\label{lsharp1}
$L^{\sharp}$ is continuous in all variables and there exist constants $c_0,c_1>0$ and $v_0\in\Rr^d$ such that
$L^\sharp(x, v, t)\geq -c_0$, and
$L^\sharp(x, v_0, t)\leq c_1$.
\item\label{lsharp2}
$L^\sharp$ is coercive in $v$:
$$
\lim_{|v|\to\infty}\frac{L^\sharp(x, v, t)}{|v|}=\infty,
$$
uniformly in $x$.
\item\label{lsharp3}
There exists a constant $C>0,$ such that
$|D_xL^\sharp|,|D_vL^\sharp|,|D^2_{vv}L^\sharp|,|D^2_{vx}L^\sharp|,|D^2_{xx}L^\sharp|\leq CL^\sharp+C$.
\item\label{lsharp4}
$H^\sharp$ is twice differentiable in $p,x$ with bounded derivatives:
$|D^2_{pp}H^\sharp|, |D^2_{px}H^\sharp|\leq C.$
\item\label{lsharp5}
$ D^2_{pp}H^\sharp( x, p,t)$ is uniformly Lipschitz in $ x, p,t$.
\end{enumerate}

To obtain the absolute continuity of the law
for extended mean-field games,
we apply the next theorem to
\[
L^\sharp(x,v,t)=L(x, v, \bX(t), \bdX(t))
\]
and
\[
H^\sharp( x, p,t)=H( x, p,\bX(t), \bdX(t)).
\]
Assumptions \ref{lsharp1}-\ref{lsharp3} are implied by
the hypotheses in the previous section for any solution $\bX$  to \eqref{mfge}. The last two, however, do not follow from the ones in the prior section.  Nevertheless, for $\bX\in C^{1,1}([0,T];L^q(\Omega))$ and the Lagrangians
\eqref{lagrangians} and Hamiltonians \eqref{hamiltonians}, the corresponding Lagrangians and Hamiltonians, $L^\sharp$ and $H^\sharp$, satisfy \ref{lsharp1}-\ref{lsharp5}.

\begin{teo}
\label{abscont} Assume Assumptions \ref{lsharp1}-\ref{lsharp5} hold. Let $u$ be a globally Lipschitz viscosity solution of the Hamilton-Jacobi equation:
\begin{equation}\label{hsharp}
\begin{cases}
-u_t+H^\sharp(x, Du, t)=0\\
u(x,T)=\psi(x),
\end{cases}
\end{equation}
where $\psi:\Rr^d\to\Rr$ is bounded and Lipschitz.
Suppose $X_0\in L^q(\Omega)$ has an absolutely continuous law with respect to the Lebesgue measure. 
Let $\bX\in C^1([0,T],L^q(\Omega))$ solve
\begin{equation}
\label{feedback}
\begin{cases}
\bdX(t)=-D_pH^\sharp( \bX, Du(\bX,t), t)\\
\bX=X_0.
\end{cases}
\end{equation}
Then, under the above conditions on $H^{\sharp}$, for every
$t<T$, $\bX(t)$ has an absolutely continuous law with respect to the Lebesgue measure.
\end{teo}

{\sc Remark:}
As above, \eqref{feedback} has a well-defined solution
since $X_0$ has an absolutely continuous law. The key point is the absolute continuity
of the law of $\bX$.

Before proving the theorem, we establish an auxiliary lemma:

\begin{lemma}
\label{zinvlip}
Let $\bz,\bq\in C^1([0,T];\Rr^d)$ and $A,B\in C([0,T];\Rr^{d\times d})$ be such that
\[
\bdz (t)=-A(t)\bq (t)-B(t)\bz (t).
\]
Suppose further there exist constants $C,\theta>0$ such that
\begin{equation}
\label{bounds}
\begin{split}
&|\bz(t)|,|\bq(t)|\leq C,\quad
\bz(t)\cdot\bq(t)\leq C|\bz(t)|^2,\\
&|A(t)|,|B(t)|\leq C,\quad
A(t)\geq \theta I,\\
&|A(t_2)-A(t_1)|\leq C|t_2-t_1|.
\end{split}
\end{equation}
Then, there exists a constant $C_1>0$, which only depends on $C,\theta$ and $T$, such that
\[
|\bz(t)|\geq C_1|\bz(0)|,\quad \forall t\in [0,T].
\]
\end{lemma}
\begin{proof}
Since $A$ is Lipschitz in $t$ with Lipschitz constant $C,$ there exist matrices $A_n\in C^1([0,T];\Rr^d)$ with $A_n\geq \theta I$, $\|A_n\|\leq C$ and $\|\frac{d}{dt}A_n(t)\|\leq C$ such that $A_n\to A$ uniformly in $[0,T]$. Let $\bz_n\in C^1([0,T];\Rr^d)$ solve the equation
\begin{equation}
\label{eqforz_n}
\begin{cases}
\bdz_n (t)=-A_n(t)\bq (t)-B(t)\bz_n (t),\\
\bz_n(0)=\bz(0).
\end{cases}
\end{equation}
By standard ODE arguments, we have $|\bz_n-\bz|\leq \epsilon_n$, where $\epsilon_n\to 0$.

Since $A_n\geq \theta I$, $\|A_n\|\leq C$, $\|\frac{d}{dt}A_n(t)\|\leq C$, there exist constants $\theta',C'>0$ such that $A^{-1}_n\geq \theta' I$, $\|A^{-1}_n\|\leq C'$
and $\|\frac{d}{dt}A^{-1}_n(t)\|\leq C'$. In this way, from \eqref{eqforz_n} one gets
\begin{equation}
\begin{split}
&\bz_n\cdot A_n^{-1}\bdz_n=-\bz_n\cdot\bq-\bz_nA^{-1}_nB\bz_n=-(\bz_n-\bz)\cdot\bq-\bz\cdot\bq-\bz_n\cdot A^{-1}_nB\bz_n\geq\\&
-(\bz_n-\bz)\cdot\bq-C|\bz|^2-\bz_n\cdot A^{-1}_n B\bz_n\geq -C\epsilon_n-C|\bz_n|^2.
\end{split}
\end{equation}
Therefore,
\[
\frac{d}{dt}(\bz_n\cdot A_n^{-1}\bz_n)\geq -C\epsilon_n-C|\bz_n|^2-\bz_n\cdot (\frac{d}{dt}A_n^{-1})\bz_n\geq
-C\epsilon_n-C|\bz_n|^2 -\frac{C}{\theta'}\bz_n\cdot A_n^{-1}\bz_n.
\]
Gronwall's inequality implies
\[
\bz_n(t)\cdot A_n^{-1}(t)\bz_n(t)\geq -C''\epsilon_n+C'''e^{-C'''t}\bz_n(0)\cdot A_n^{-1}(0)\bz_n(0).
\]
Passing to the limit when $n\to\infty$ yields
\[
\bz(t)\cdot A^{-1}(t)\bz(t)\geq C'''e^{-C'''t}\bz(0)\cdot A^{-1}(0)\bz(0).
\]
The previous estimate together with $A^{-1}_n\geq \theta' I$, $\|A^{-1}_n\|\leq C'$ implies
\[
|\bz(t)|\geq C_1|\bz(0)|,
\]
for some constant $C_1>0$.
\end{proof}

Now, we proceed to the proof of Theorem \ref{abscont}:
\begin{proof}
Due to the assumptions on $L^{\sharp}$ and $\psi$, the value function
\[
u^{\sharp}(x,t)=\inf_{\bx:\ \bx(t)=x}\int\limits_t^TL^\sharp(\bx, \bdx, s)ds+\psi(\bx(T)),
\]
is a viscosity solution of \eqref{hsharp}. Moreover, under the hypothesis on $L^{\sharp}$, arguing as in the proof of Lemma \ref{blsmc}, we have that $u^{\sharp}$ is uniformly bounded, Lipschitz, and semiconcave in any interval $[0,T'],\,T'<T$. Hence, it is the unique viscosity solution to \eqref{hsharp}, so $u=u^{\sharp}$.

The optimal trajectories starting at $(x,0)$ for every point of differentiability $x$ of $u$ are given by the Hamiltonian flow:
\begin{equation}
\label{hflow}
\begin{cases}
\bdx(x,t)=-D_pH^\sharp(\bx, \bp, t)\\
\bdp(x,t)=D_xH^\sharp(\bx, \bp, t)\\
\bx(x,0)=x,\,\bp(x,0)=Du(x,0).
\end{cases}
\end{equation}
Furthermore, along these trajectories, $(\bx(x,t),t),\,t>0$, $u$ is differentiable and $\bp(x,t)=Du(\bx(x,t),t)$.
Fix a compact set $K\subset\Rr^d$. Let $x,y\in K$ points for which the flow is defined. Set
$\bz(t)=\bx(x,t)-\bx(y,t)$ and $\bq(t)=\bp(x,t)-\bp(y,t)$. Then
\[
\bdz (t)=-A(t)\bq (t)-B(t)\bz (t),
\]
where
$$A(t)=\int_0^1D^2_{pp}H^\sharp(\tau\bx(x,t)+(1-\tau)\bx(y,t), \tau\bp(x,t)+(1-\tau)\bp(y,t), t)d\tau$$
and
$$B(t)=\int_0^1D^2_{px}H^\sharp(\tau\bx(x,t)+(1-\tau)\bx(y,t), \tau\bp(x,t)+(1-\tau)\bp(y,t),t)d\tau.$$
Equation \eqref{hflow} implies that there exists a constant $C_K$ such that
the trajectories $\bx(x,t),\bp(x,t)$ are $C_K$-Lipschitz in $t$ and $|\bz|,|\bq|\leq C_K$. The assumptions on $H^\sharp$ imply that $A(t)$ is Lipschitz with Lipschitz constant $C_K$, provided $C_K$ is large enough.
In addition, $A(t)\geq\theta I$.
Since $u$ is semiconcave in $x$ uniformly for $t\in[0,T']$, we have
\[
\bz(t)\cdot\bq (t)=(\bx(x,t)-\bx(y,t), Du(\bx(x,t),t)-Du(\bx(x,t),t))\leq C|\bx(x,t)-\bx(y,t)|^2=C|\bz(t)|^2.
\]
Thus $\bz$ and $\bq$ satisfy the conditions of Lemma \ref{zinvlip}. Hence  $|\bx(x,t)-\bx(y,t)|\geq C_K(T')|x-y|,\,t\in[0,T']$. Therefore, the mapping $x\mapsto\bx(x,t)$ is invertible on the set where it is defined. 
Moreover, for any compact $K\subset\Rr^d$, the inverse of the map $x\in K\to \bx(x,t)$ is Lipschitz.

Let $A\in\Rr^d$ be a set of Lebesgue measure zero. Set $B=(\bx)^{-1}(\cdot, t)(A)$. For any
compact $K\subset\Rr^d$, the inverse of the map $x\in K\to \bx(x,t)$ is Lipschitz. Hence, $B\cap K$ has Lebesgue measure zero. Therefore, $B$ has Lebesgue measure zero.
Since $\bx$ is defined a.e., we have that
$\bX(t)=\bx(X_0,t)$ a.s.. So $X_0=(\bx)^{-1}(\cdot, t)(\bX(t))$ a.s.. Consequently, $P(\bX(t)\in A)=P(X_0\in B)=0$. Accordingly, $\bX(t)$ has an absolutely continuous law.
\end{proof}

\begin{cor}
Assume that $H,\Psi, X_0$ satisfy Assumptions \ref{x0}-\ref{DL}. Further suppose that
\begin{enumerate}
\item
$H$ is twice differentiable in $p,x$ with bounded derivatives:
$|D^2_{pp}H|, |D^2_{px}H|\leq C(X,Z).$
\item
$ D^2_{pp}H( x, p,X, Z)$ is uniformly Lipschitz in $ x, p,X,Z.$
\end{enumerate}
Then for any solution $(u,\bX)$ of \eqref{mfge}(in the sense of Theorem \ref{teoexist}),  $\bX(t),\ t<T,$ has an absolutely continuous law.
\end{cor}
The Corollary follows from Theorem \ref{abscont} for the Hamiltonian $H^\sharp( x, p,t)=H( x, p,\bX(t), \bdX(t)).$

\section{Uniqueness}
\label{uniq}
In this section, we discuss two approaches to the uniqueness problem. First, we illustrate how to adapt Lasry-Lions monotonicity argument for the random variable framework. However, this proof applies only  to
classical
solutions (though it may prove possible to extend it to viscosity solutions if the mean-field
trajectories admit an absolutely continuous law). Next, we consider a second uniqueness technique that gives uniqueness for a more general class of problems without any further conditions on solutions. 
This method generalizes, even for classical mean-field games, prior
results in the literature.

\subsection{Lasry-Lions monotonicity argument}
\label{monot}
In this section, we consider a version of Lasry-Lions monotonicity method to prove uniqueness for extended mean-field games. The original idea can be explained
as follows: let $(\theta, u)$ and $(\tilde \theta, \tilde u)$
be two distinct solutions of \eqref{mfgo}. Monotonicity conditions on $H$ give
\[
\frac{d}{dt}\int (\theta-\tilde \theta) (u-\tilde u)> 0.
\]
Hence, $\int (\theta-\tilde \theta) (u-\tilde u)$
is strictly monotone in time.
However, this quantity also vanishes at $t=0$ and $t=T$. This establishes a contradiction. 
For a more detailed argument, see \cite{ll1} or the notes \cite{cardaliaguet}.

In the current setting, given two solutions $(X, u)$ and $(\tilde X, \tilde u)$,
it suffices to show that 
\[
 E \left(u(X(t), t) -\tilde u(X(t), t)+\tilde u(\tilde X(t),t)-u(\tilde X(t),t)\right),
\]
is strictly monotone. 
To simplify, we assume that
\begin{enumerate}
\item
\label{h0v}
\[
H( x, p, X, Z)=H_0(x, p+\beta EZ)+V(x, X),
\]
with $\beta\geq 0$.
\item 
 $H_0(x,p)$ is convex in $p$.
\item Monotonicity Condition:
\label{montcnd}
\[
E\left(V(X, X)-V(X, \tilde X)+V(\tilde X, \tilde X)-V(\tilde X, X)\right) <0,
\]
if $Law (X)\neq Law (\tilde X)$. For all $X,\tilde X\in L^q(\Omega)$.
\item
\label{psimon}
$\psi$ satisfies
\[
E\left(\psi(X, X)-\psi(X, \tilde X)+\psi(\tilde X, \tilde X)-\psi(\tilde X, X)\right) \geq 0,
\]
for all $X,\tilde X\in L^q(\Omega)$.
\suspend{enumerate}

\begin{teo}
\label{unq1}
Assume $H$ satisfies conditions \ref{h0v}-\ref{montcnd}, $\psi$ satisfies \ref{psimon}. Then, there exists at most one (classical, in the sense of the definition in Theorem \ref{teoexist})
solution $(u,\bX)$  to \eqref{emfg}.
\end{teo}
\begin{proof}
Let $(u,\bX)$ and $(\tilde u,\tilde \bX)$ be two solutions of \eqref{emfg}. Then we have
\begin{align*}
&\frac{d}{dt} \left( u(\bX(t), t)-\tilde u(\bX(t),t)\right)=H_0(\bX(t), D_xu(\bX(t),t)+\beta E\dot \bX )-\\
&H_0(\bX(t), D_x\tilde u(\bX(t),t)+\beta E \dot{\tilde \bX})+V(\bX(t), \bX(t))-V(\bX(t), \tilde \bX(t))+\\&\dot X\cdot  (D_x u(\bX(t), t)-D_x\tilde u(\bX(t),t))\leq (D_pH_0( \bX(t),D_xu(\bX(t),t)+\beta E\dot \bX)\\&+\dot \bX\cdot (D_x u(\bX(t), t)-D_x\tilde u(\bX(t),t))+\beta\dot \bX \cdot ( E \dot{\tilde \bX}- E \dot  \bX)+V(\bX(t), \bX(t))-V(\bX(t), \tilde \bX(t)).
\end{align*}
We add a similar expression for $\frac{d}{dt} \tilde u(\tilde \bX(t), t)-u(\tilde \bX(t),t)$ and
obtain
\begin{align*}
\frac{d}{dt} &\left( u(\bX(t), t)-\tilde u(\bX(t),t)+ \tilde u(\tilde \bX(t), t)-u(\tilde \bX(t),t)\right)\\
=&V(\bX(t), \bX(t))-V(\bX(t), \tilde \bX(t))+V(\tilde \bX(t), \tilde \bX(t))-V(\tilde \bX(t), \bX(t))\\
&+\beta (\dot \bX-\dot {\tilde \bX}) E (\dot{\tilde \bX}-E \dot  \bX).
\end{align*}
By taking the expectation and using both the monotonicity condition and $\beta\geq 0$, we get that
\begin{equation}\label{sharp}
\frac{d}{dt}E \left( u(\bX(t), t)-\tilde u(\bX(t),t)+ \tilde u(\tilde \bX(t), t)-u(\tilde \bX(t),t)\right)<0.\tag{$\star$}
\end{equation}

This is a contradiction since
\[
E\left( u(\bX(0), 0)-\tilde u(\bX(0),0)+ \tilde u(\tilde \bX(0), 0)-u(\tilde \bX(0),0)\right)=0,
\]
and
\[
E\left( u(\bX(T), T)-\tilde u(\bX(T),T)+ \tilde u(\tilde \bX(T), t)-u(\tilde \bX(T),T)\right)=\]
\[
E\left(\psi(\bX(T), \bX(T))-\psi(\bX(T), \tilde \bX(T))+\psi(\tilde \bX(T), \tilde \bX(T))-\psi(\tilde \bX(T), \bX(T))\right) \geq 0.
\]
\end{proof}

\begin{remark}
Assumption \ref{h0v} is not essential, but it  simplifies the computations and the remaining conditions substantially. To obtain a more general condition for uniqueness one can compute the expression on the left-hand side of \eqref{sharp} in terms of $Du$, $D\tilde u$, $X$, and $\tilde X$ and require the resulting expression to be positive, unless $Du=D\tilde u$ and $Law(\bX, \dot \bX)=Law(\tilde \bX, \dot \tilde\bX).$ However, a general condition is easier to write in terms of Lagrangians and is considered in the next section.
\end{remark} 

To prove the uniqueness of solutions $(u,\bX)$, if $u$ is not everywhere differentiable, we assume further
\resume{enumerate}
\item
$X_0\in L^q(\Omega)$ has an absolutely continuous law with respect to the Lebesgue measure;
\item
$\psi$ is bounded and Lipschitz in $x$;
\item
\label{Vbd}
$V$ is $C^2$ bounded in $x$;
\item
\label{dL0bd}
There exists a constant $C>0$, such that
$|D_vL_0|,|D^2_{vv}L_0|\leq CL_0+C$, where $L_0$ is the Legendre transform of $H_0$;
\item
\label{H0bd}
$D^2_{pp}H_0$ is Lipschitz continuous.
\suspend{enumerate}

\begin{teo}
\label{unq1'}
Suppose Conditions \ref{h0v}-\ref{H0bd} hold. Then, there exists at most one solution $(u,\bX)$  to \eqref{emfg} with $\bX\in C^1([0,T],L^q(\Omega))$ and $\bdX$ Lipschitz in $t$.
\end{teo}
\begin{proof}
Let $(u,\bX)$ and $(\tilde u,\tilde \bX)$ be two solutions of the system \eqref{emfg} with $\bX, \tilde \bX\in C^1([0,T],L^q(\Omega))$ and $\bdX,\dot{\tilde\bX}$ Lipschitz in $t$.
It is easy to check that under the conditions on $H_0$ and $V$ the Hamiltonians
$\widetilde{H}_1( x, p,t)=H( x, p,\bX(t),\bdX(t))$ and $\widetilde{H}_2( x, p,t)=H( x, p,\tilde\bX(t),\dot{\tilde\bX}(t))$ satisfy the conditions of the Theorem \ref{abscont}.
Thus, we conclude that $\bX(t)$ and $\tilde \bX(t),$ $t<T$ have absolutely continuous laws with respect to the Lebesgue measure. Because $u,\tilde u$ are Lipschitz, $Du(\bX(t),t)$, $D u(\tilde\bX(t),t)$, $D\tilde u(\bX(t),t)$, $D\tilde u(\tilde\bX(t),t)$ are well defined almost surely. Therefore, the arguments from the proof of Theorem \ref{unq1} hold without changes.
\end{proof}

\subsection{The second approach for uniqueness}
\label{uniq2}

In this section, we discuss another method to prove the uniqueness of solutions of \eqref{mfge}.
This approach is valid even if the law of $\bX$ is not absolutely continuous and extends Theorem \eqref{unq1'}. Consequently, it leads to a more general uniqueness result.
We will assume that $L$ satisfies the following monotonicity condition:
\resume{enumerate}
\item \label{Lmon}
\[E\bigg(L(X,Z,X,Z)-L(\tilde X, \tilde Z, X, Z)+L(\tilde X, \tilde Z, \tilde X, \tilde Z)-L(X, Z, \tilde X,\tilde Z) \bigg)>0\]
if $Law (X, Z)\neq Law(\tilde X, \tilde Z)$.
\suspend{enumerate}

\begin{teo}
\label{unq2}
Under Assumptions \ref{x0}-\ref{DL}, \ref{psimon} and \ref{Lmon}  there exists at most one solution to \eqref{mfge}.
\end{teo}
\begin{proof}
Suppose $(\bX, u)$ and $(\tilde \bX, \tilde u)$ are two solutions of \eqref{mfge}. For a.e. $\omega$,
 $\bX(\omega)$ and $\tilde \bX(\omega)$ are optimal trajectories for the optimal control problems
with Lagrangians $L(x, v, \bX,\dot \bX)$ and $L(x, v, \tilde \bX, \dot{\tilde \bX})$ and terminal values $\psi(x,\bX(T))$ and $\psi(x,\tilde \bX(T))$, respectively. Since $u,\tilde u$ are the respective value functions, we have
\[
u(\bX(0),0)=\int\limits_0^TL(\bX(s), \dot \bX(s), \bX(s),\dot \bX(s))ds+ \psi(\bX(T),\bX(T)),
\]
\[
u(\tilde \bX(0),0)\leq\int\limits_0^TL(\dot {\tilde \bX(s), \tilde \bX}(s), \bX(s), \dot \bX(s))ds+ \psi(\tilde \bX(T), \bX(T)),
\]
\[
\tilde u(\tilde \bX(0),0)=\int\limits_0^TL(\tilde \bX(s), \dot {\tilde \bX}(s),\tilde \bX(s),\dot {\tilde \bX}(s))ds+ \psi(\tilde \bX(T),\tilde \bX(T)),
\]
and
\[
\tilde u(\bX(0),0)\leq\int\limits_0^TL(\bX(s), \dot \bX(s),\tilde \bX(s),\dot {\tilde \bX}(s))ds+ \psi(\bX(T),\tilde \bX(T)).
\]
By combining the previous expressions, we get the inequality
\begin{align}
0=E\left( u(\bX(0), 0)-\tilde u(\bX(0),0)+ \tilde u(\tilde \bX(0), 0)-u(\tilde \bX(0),0)\right) \geq & \nonumber\\
\int\limits_0^T E\bigg(L(\bX(s), \dot \bX(s), \bX(s),\dot \bX(s))-L(\tilde \bX(s), \dot {\tilde \bX}(s), \bX(s),\dot \bX(s))+ & \nonumber\\
L(\tilde \bX(s), \dot {\tilde \bX}(s), \tilde \bX(s),\dot {\tilde \bX}(s))-L(\bX(s), \dot \bX(s), \tilde \bX(s), \dot {\tilde \bX}(s)) \bigg) ds + & \nonumber\\
E\left(\psi(\bX(T),\bX(T))-\psi(\tilde \bX(T),\bX(T))+\psi(\tilde \bX(T),\tilde \bX(T))-\psi(\bX(T),\tilde \bX(T))\right). &
\end{align}
Conditions \ref{Lmon}, \ref{psimon}, and the preceding inequalities imply $\bX(s)=\tilde \bX(s)$. Then, the uniqueness of viscosity solutions yields $u=\tilde u$.
\end{proof}
The monotonicity condition \ref{Lmon} is implied by a condition that depends only on the second derivatives of $L$.
Indeed, we have
\begin{align*}
E\bigg(L(X,Z,X,Z)-L(\tilde X, \tilde Z, X, Z)+L(\tilde X, \tilde Z, \tilde X, \tilde Z)-L(X, Z, \tilde X,\tilde Z) \bigg)= & \\
\int_0^1\int_0^1 E[(Z-\tilde Z)^TD^2_{vZ}L\cdot(Z-\tilde Z)+(X-\tilde X)^TD^2_{xX}L\cdot(X-\tilde X)+& \\
(Z-\tilde Z)^TD^2_{vX}L\cdot(X-\tilde X)+(X-\tilde X)^TD^2_{xZ}L\cdot (Z-\tilde Z)]d\tau d\theta.
\end{align*}
Here, all the derivatives of $L$ are evaluated at $(Z_\tau,X_\tau,Z_\theta,X_\theta)$, with
$X_\theta=(1-\theta)X+\theta\tilde X$ and $Z_\theta=(1-\theta)Z+\theta\tilde Z.$\\
Consequently, uniqueness holds if
\[
E[Z^TD^2_{vZ}L\cdot Z+Y^TD^2_{xX}L\cdot Y+
Z^TD^2_{vX}L\cdot Y+Y^TD^2_{xZ}L\cdot Z]>0, \text{ for all } (Y,Z)\neq 0,
\]
where the derivatives are evaluated at an arbitrary point $(A,B,C,D)\in (L^q(\Omega))^4$.

It is easy to see that, in the proof of Theorem \ref{unq2}, we use a weaker version of condition
\ref{Lmon}, namely
\resume{enumerate}
\item \label{Lmon'}
\[E\bigg(L(X,Z,X,Z)-L(\tilde X, \tilde Z, X, Z)+L(\tilde X, \tilde Z, \tilde X, \tilde Z)-L(X, Z, \tilde X,\tilde Z) \bigg)\leq0,\]
if and only if $L(x, v, X,Z)=L(x, v, \tilde X,\tilde Z)$ for all $v,x\in\Rr^d$.
\end{enumerate}

An example that satisfies \ref{Lmon'} is:
\[
L(x, v, X,Z)=L_0(v)+\beta vEZ - V(x,X),
\]
where $L_0$ is strictly convex, $\beta\geq 0$ and $V$ satisfies the monotonicity condition \ref{psimon}. The corresponding Hamiltonian for this Lagrangian is
\[
H( x, p,Z)=H_0(p+\beta EZ)+V(x,X)
\]
where $H_0=L_0^*$. Thus, the uniqueness  result in this section generalizes the result in the previous section.

Another example of a Lagrangian that satisfies condition \ref{Lmon'} is
\[
L(x, v, X,Z)=\frac{|v+\beta EZ|^2}{2}-V(x,X),
\]
where $V$ satisfies the monotonicity condition \ref{psimon}.

\section{Final Remarks}
\label{chp: furthprob}

In this concluding section, we briefly discuss the master 
equation formulation for deterministic mean-field games. A probabilistic approach for mean-field problems with common noise was examined extensively in 
\cite{Carmona3}.

\subsection{Master Equation in the deterministic case}

In this section, we consider the so-called {\em Master Equation}  for deterministic MFG's. The Master Equation for mean-field games was introduced by Lions in his Coll\'ege de France lectures.
Let $(\Omega,\, \Fff, P)$ be a probability space, where $\Omega$ is an arbitrary set,
$\Fff$ is a $\sigma$-algebra on $\Omega$, and $P$ is a probability measure.
We start by considering the optimal control problem:
\begin{equation}
\label{gvalfunc}
V(x,Y,t)=\inf\limits_{\bv}\left[\int_t^TL(\bx(s), \bv(s), \bY(s), \dot \bY(s))ds+\psi(\bx(T),\bY(T))\right],
\end{equation}
where $\bx$ is the trajectory of a player starting at time $t$ at the point $\bx(t)=x$, controlled by
$\dot\bx=\bv$.
$\bY(\cdot)$ is the trajectory of the population of the players who move along a vector field $$b\colon L^q(\Omega;\Rr^d)\times[0,+\infty)\to L^q(\Omega;\Rr^d).$$
More precisely, if the random variable  $Y$ corresponds to the initial states of the population, 
its evolution is given by
\begin{equation}
\label{YbY}
\dot \bY=b(\bY,t),\quad \bY(0)=Y.
\end{equation}
We define
$L^q_{ac}(\Omega, \Rr^d)$ to be the subspace of $L^q(\Omega, \Rr^d)$, consisting
of random variables that have absolutely continuous laws. We assume  the vector field $b$ is such that for any $Y\in L^q_{ac}(\Omega;\Rr^d)$, equation \eqref{YbY} has a unique solution in
$L^q(\Omega;\Rr^d)$.
It follows from above that $V$ is well defined and is a viscosity solution of the Hamilton-Jacobi equation
\begin{equation}
\label{ocpp}
-V_t-D_YV(x,Y,t)\cdot b(Y,t)+H(x, D_xV(x,Y,t),Y,b(Y,t))=0,
\end{equation}
where
\[
H( x, p,Y,Z)=\sup\limits_v\{\,-p\cdot v-L(x, v, Y,Z)\,\}.
\]
If $V$ is a classical solution to \eqref{ocpp},
the optimal control is given in the feedback form by $v^*=-D_pH(x, D_xV(x,Y,t),Y,\dot Y)$.

As before,
we assume all players act rationally. Then
each of them follows the optimal flow, that is
\[
\dot \bY=-D_pH(\bY, D_xV(\bY,\bY,t),\bY, \dot \bY).
\]
We assume one can solve the previous equation with respect to $\dot \bY$ as $\dot \bY= G(D_xV(\bY,\bY,t),\bY)$. Hence, we have $b(Y,t)=G(D_xV(Y,Y,t),Y)$.
Thus, we end up with the equation
\begin{equation}
\label{master'}
\begin{cases}
-V_t(x,Y,t)+D_YV(x,Y,t)\cdot G(D_xV(Y,Y,t),Y)+H(x, D_xV(x,Y,t),Y,b(Y,t))=0,\\
V(x,Y,T)=\psi(x,Y).
\end{cases}
\end{equation}
We call this the {\em master equation}.

In general, first-order PDEs do not admit classical solutions. Consequently, we
must look, for instance, for solutions that are differentiable
almost everywhere with respect to the variable x. Therefore, one can make sense of $G(D_xV(Y,Y,t),Y)$ provided $Y$
has an absolutely continuous law. For this reason, we work in the space $L^q_{ac}(\Omega;\Rr^d)$.
\begin{df}
Let $V\colon\Rr^d\times L^q_{ac}(\Omega;\Rr^d)\times [0,T]\to\Rr^d$ be a continuous function, Lipschitz continuous in the first variable, and $b\colon L^q_{ac}(\Omega;\Rr^d)\times [0,T]\to\Rr^d$ a vector field. We say that the couple $(V,b)$ is a solution to \eqref{master'} if
\begin{itemize}
\item
$V$ is a viscosity solution of
\begin{equation*}
\label{mast}
\begin{cases}
-V_t(x,Y,t)+D_yV(x,Y,t)\cdot b(Y,t)+H(x, D_xV(x,Y,t),Y,b(Y,t))=0,\\
V(x,Y,T)=\psi(x,Y),
\end{cases}
\end{equation*}
that is, for any continuous function $\phi\colon\Rr^d\times L^q(\Omega)\times[0,T]\to\Rr$ and any point $(x,Y,t)\in \argmax V-\phi$ (resp. argmin) where $\phi$ is differentiable
\[
-\phi_t(x,Y,t)+D_y\phi(x,Y,t)\cdot b(Y,t)+H(x, D_x\phi(x,Y,t),Y,b(Y,t))\leq0\; (\text{resp. } \geq 0).
\]
\item
$b(Y,t)=G(D_xV(Y,Y,t),Y)$ a.s. for any $Y\in  L^q_{ac}(\Omega;\Rr^d)$.
\end{itemize}
\end{df}

\vskip0.3cm
Next, we assume that $L$ and $\psi$ satisfy the following hypotheses:
\begin{enumerate}
\item
\label{psi'}
$\psi$ is bounded in both variables and Lipschitz  in $x$:
\[
|\psi(x_1,X)-\psi(x_2,X)|\leq C|x_1-x_2|,\quad \forall x_1, x_2\in\Rr^d.
\]
\item
There exist constants $c_0,c_1>0$ and a vector function $v_0\colon L^q(\Omega;\Rr^d)\times L^q(\Omega;\Rr^d)\to \Rr^d$ such that
\[L(x, v, X,Z)\geq-c_0\]
and
\[
L(x, v_0(X,Z), X,Z)\leq c_1
\]
for all $x,v\in\Rr^d,X,Z\in L^q(\Omega;\Rr^d).$
\item
\label{L'}
$L$ is twice differentiable in $x,v$, and we have the following bounds
\[
|D_xL(x, v, X,Z)|, |D^2_{xx}L(x, v, X,Z)|,|D^2_{xv}L(x, v, X,Z)|,|D^2_{vv}L(x, v, X,Z)|\leq CL(x, v, X,Z)+C
\]
for all $x,v\in\Rr^d,X\in L^q(\Omega;\Rr^d)$.
\end{enumerate}

\begin{pro}
Assume that \ref{psi'}-\ref{L'} hold. Then the function $V$ defined in \eqref{gvalfunc} for a fixed vector field $b$ is finite, bounded, Lipschitz and semiconcave in $x$:
\begin{enumerate}
\item
There exists a constant $C$ such that
\[|V(x,Y,t)|\leq C,\quad \forall t\in[0,T],\,x,h\in\Rr^d, Y\in L^q(\Omega;\Rr^d).\]
\item
There exists a constant $C$ such that
\[
|V(x+h,Y,t)-V(x,Y,t)|\leq C|h|,\quad \forall t\in[0,T],\,x,h\in\Rr^d, Y\in L^q(\Omega;\Rr^d).
\]
\item
For any $t,t<T$ there exists a constant $C(t)$ such that
\[V(x+h,Y,t)+V(x-h,Y,t)-2V(x,Y,t)\leq C(t)|h|^2,\quad \forall x,h\in\Rr^d, Y\in L^q(\Omega;\Rr^d).\]
\end{enumerate}
Moreover, the constants are uniform in $b.$
\end{pro}

\begin{proof}
The proof is standard and follows the arguments in Section 4.
\end{proof}
In the following, we will present two simple results. These make the connection between master equation and extended mean-field game equation \eqref{mfge} from Section \ref{exist}. Let $V$ be a classical solution to \eqref{master'}, and let $\bX(\cdot)$ be a solution to the ODE
\[
\begin{cases}
\dot \bX=-D_pH(\bX, D_xV(\bX,\bX,t), \bX,\bdX)\\
\bX(0)=X_0.
\end{cases}
\]
\begin{pro}
Let $u(x,t)=V(x,\bX(t),t)$, then the pair $(u,\bX)$ solves equation \eqref{mfge}.
\end{pro}
\begin{proof}
Since $V$ is smooth, we have
\[
D_xu(x,t)=D_xV(x,\bX(t),t),
\]
and
\[
u_t(x,t)=V_t+D_yV\cdot\dot X=V_t(x,\bX(t),t)-D_yV(x,\bX(t),t)\cdot D_pH(\bX(t), D_xV(\bX(t),\bX(t),t), \bX(t),\bdX(t)).
\]
Plugging this in \eqref{master'} with $Y=\bX(t)$  we get
\[
-u_t(x,t)+H( x, D_xu(x,t),\bX(t),\bdX(t))=0.
\]
Additionally, we have $u(x,T)=V(x,\bX(T),T)=\psi(x,\bX(T))$.
\end{proof}

Now we assume that $\psi$ and $L$ satisfy conditions \ref{psi}-\ref{hjuniqc}. Because of this, by Theorem \ref{teoexist}, for any $Y\in L^q_{ac}(\Omega,\Rr^d)$  there exist solutions $(u(x,s,t),\bX(s,t))$ to
\begin{equation}
\label{emfg*}
\begin{cases}
-u_s(x,s,t)+H(x, D_xu(x,s,t), \bX(s,t), \frac{\partial\bX}{\partial s}(s,t))=0\\
\frac{\partial\bX}{\partial s}(s,t)=-D_pH(\bX(s,t), D_xu(\bX(s,t),s),\bX(s,t), \frac{\partial\bX}{\partial s}(s,t))\\
u(x,T,t)=\psi(x,\bX(T)),\,\bX(t,t)=Y,
\end{cases}
\end{equation}
where we set $\widetilde{V}(x,Y,t)=u(x,t,t)$.

\begin{pro}
Assume $\psi$ and $L$ satisfy conditions \ref{psi}-\ref{hjuniqc}. Then $\widetilde{V}$ is a viscosity solution of
\begin{equation*}
\begin{cases}
-V_t(x,Y,t)+D_yV(x,Y,t)\cdot b(Y,t)+H(x, D_xV(x,Y,t),Y, b(Y,t))=0,\\
V(x,Y,T)=\psi(x,Y),
\end{cases}
\end{equation*}
for $b(Y,t)=G(D_xV(Y,Y,t),Y)$.
\end{pro}

\begin{proof}
Note that by definition of $\widetilde{V}$,  $\widetilde{V}(x,\bX(s,t),s)=u(x,s,t)$. Let $\phi\colon\Rr^d\times L^q(\Omega;\Rr^d)\times \Rr$ be a continuous function, differentiable at  $(x,Y,t)$. Suppose $\widetilde{V}-\phi$ has a local maximum at $(x,Y,t)$. Let $\varphi(\cdot,s)=\phi(\cdot,X(s,t),s)$. Hence $\varphi$ is a continuous function that is differentiable at  $(x,t)$, and  $u-\varphi$ has a local maximum at $(x,t)$. Since $u$ is a viscosity solution of the Hamilton-Jacobi equation, we have
\[
-\varphi_t(x,t)+H(x, D_x\varphi(x,t),\bX(t,t),\frac{\partial\bX}{\partial s} (t,t))\leq 0.
\]
Because
$\varphi_t(x,t)=\phi_t(x,\bX(t, t),t)+D_y\phi(x,\bX(t, t),t)\cdot \frac{\partial\bX}{\partial s}(t, t)=\phi_t(x,Y,t)+D_y\phi(x,Y,t)\cdot b(Y,t) $
and
$D_x\varphi(x,t)=D_x\phi(x,Y,t)$, we get
\[
-\phi_t(x,Y,t)-D_y\phi(x,Y,t)\cdot b(Y,t)+H(x, D_x\phi(x,Y,t),Y,b(Y,t))\leq 0.
\]
This proves that $\widetilde{V}$ is a viscosity subsolution. Similarly, we show that it is also a supersolution. For $t=T,$  $\bX(\cdot
,T)\equiv Y$ thus, $\widetilde{V}(x,Y,T)=u(x,T,T)=\psi(x,Y).$

\end{proof}

\subsection{Conclusion}

The random variable framework developed in this paper is a powerful tool for studying of mean-field games. It allows the formulation of extensions of the original problem in which each player takes into consideration the actions of the rest of the population. We obtained new existence and uniqueness results that extend those presently available. We provide conditions for the absolute continuity of the law and exhibited two problems for which exact solutions can be computed.  
 
The master equation was introduced by P.L. Lions in his lectures at College de France. It makes it possible to study mean-field game problems using a single equation, making it an important aspect of the general theory of mean-field games. The master equation and its solution can be defined within the framework of this paper. Propositions 3 and 4 illustrate the connection between solutions of the master equation and the extended mean-field games system. Finally, the master equation can be used to formulate games with correlations between the players caused by a common Brownian noise without using backward-forward stochastic differential equations. Recent substantial progress on this subject was achieved in \cite{ Carmona3, Carmona2, Carmona1}. Nevertheless, many questions remain unanswered, and we regard this as an important direction that should be pursued further.   

\appendix

\section*{Appendix: Proof of Lemma \ref{blsmc}}\label{appendix}

Here, we give the proof of Lemma \ref{blsmc} and determine explicitly 
the dependence of $c_1$, $c_5$, and $c_6$ on the data of the problem.

\begin{proof}
	For the first claim, note that
	$$
	\tilde{u}(x,t)\leq \int\limits_t^TL(\bx,v_0(\bX,\dot{\bX}),\bX,\dot{\bX})+\psi(\bx(T),\bX(T))\leq(T-t)c_1+\|\psi\|_{\infty}.
	$$
	
	To prove that $\tilde{u}$ is Lipschitz, take $x,y\in\Rr^d$, with $|y|\leq 1$. Let $x^*$ be the optimal trajectory at a point $(x,t)$. Such optimal trajectory exists by standard control theory arguments.  We have
	$
	\tilde{u}(x,t)= \int\limits_t^TL({\bx}^*,\dot{\bx}^*, \bX,\dot{\bX})+\psi({\bx}^*(T),\bX(T)),
	$
	and
	$
	\tilde{u}(x+y,t)\leq\int\limits_t^TL({\bx}^*+y, \dot{\bx}^*, \bX, \dot{\bX})+\psi({\bx}^*(T)+y,\bX(T)).
	$
	Let
	$
	f(\tau)=\int\limits_t^T L({\bx}^*+\tau y, \dot{\bx}^*, \bX, \dot{\bX}).
	$
	Then
	$
	f(0)=\tilde{u}(x,t)-\psi({\bx}^*(T),\bX(T))\leq C(T-t)+C,
	$
	where the constant $C$ depends only on $L, \psi$ and $T.$ Using Assumption \ref{DL}, we obtain
	$$
	f'(\tau)=\int\limits_t^T D_xL({\bx}^*+\tau y, \dot{\bx}^*,\bX,\dot{\bX})\cdot y\leq\int\limits_t^T (c_2L({\bx}^*+\tau y, \dot{\bx}^*,\bX,\dot{\bX})+c_3)|y|
	\leq (c_2f(\tau)+Tc_3)|y|.
	$$
	Consequently, by Gronwall inequality, $f(\tau)\leq C$ and $f'(\tau)\leq (C(T-t)+C)|y|$.
	Therefore,
	$$
	\tilde{u}(x+y,t)-\tilde{u}(x,t)\leq f(1)-f(0)+\psi({\bx}^*(T)+y,\bX(T))-\psi({\bx}^*(T),\bX(T))\leq (C(T-t)+C)|y|.
	$$
The previous estimate proves that $\tilde{u}$ is uniformly Lipschitz in $x.$\\
	For the semi-concavity, we take any $t_1<T,$ $t\leq t_1,$ $x,y\in\Rr^d$ with $|y|\leq 1$, $\bx^*$ as above, $\mathbf{y}(s)=y\frac{T-s}{T-t}$,
	and let
	$$
	g(\tau)=\int\limits_t^T L({\bx}^*(s)+\tau \mathbf{y}(s), \dot{\bx}^*(s)+\tau \dot{\mathbf{y}}(s),\bX(s),\dot{\bX}(s))ds.
	$$
	Using  Lemma \ref{c8} with the bounds on $DL$ from Assumption \ref{DL},  we get
	$g'(\tau)\leq C(Lip(\Phi ),T-t_1)(g(\tau)+1)$.
	Hence, by Gronwall inequality, $g(\tau)\leq C'(Lip(\Phi ),T-t_1)$.
	Similarly, using Lemma \ref{c8} with the  bounds on $D^2L$ from Assumption \ref{DL}, we get
	\begin{equation*}
	\begin{split}
	g''(\tau)
	\leq
	(1+c_4)(C g(\tau)+C) \left(C+\frac{C}{(T-t)^2}\right)|y|^2\leq
	c'|y|^2 ,
	\end{split}
	\end{equation*}
	where $c'$ depends only on $L, \psi, T,$ $Lip(\Phi)$ and $T-t_1$.
	We conclude
	$$
	\tilde{u}(x+y,t)+\tilde{u}(x-y,t)-2\tilde{u}(x,t)\leq g(1)+g(-1)-2g(0)\leq2\max\limits_{[-1,1]} g''\leq c_6|y|^2.
	$$
\end{proof}

\bibliographystyle{plain}
\bibliography{mfg}

\end{document}